\documentclass [10pt, leqno]{article}
\usepackage {graphicx}
\usepackage{wrapfig}
\usepackage{amsmath,amssymb, amsthm}
\usepackage[colorlinks,linkcolor=blue,citecolor=blue,urlcolor=blue, linktocpage=true,dvips]{hyperref}

\numberwithin{equation}{section}

\newtheorem{theorem}{Theorem}[section]
\newtheorem{proposition}[theorem]{Proposition}
\newtheorem{lemma}[theorem]{Lemma}
\newtheorem{corollary}[theorem]{Corollary}

\newtheorem{remark}[theorem]{Remark}

\def\P{\mathbf{P}}

\def\uno{\text{\bf 1}}
\def\lcm{\text{\rm lcm}}

\begin{document}

\title{Equidistribution and coprimality}
 \author{Jos\'{e} L. Fern\'{a}ndez and  Pablo Fern\'{a}ndez}

 \date{\today}

 \renewcommand{\thefootnote}{\fnsymbol{footnote}}
\renewcommand{\thefootnote}{\arabic{footnote}}

 \maketitle

 \begin{abstract}
{This paper is devoted to the study of equidistributional properties of \textit{totient points} in  $\mathbb{N}^r$, that is, of coprime $r$-tuples of integers, with particular emphasis on some relevant sets of totient points fulfilling  extra divisibility or coprimality conditions, or lying on arithmetic progressions.}\end{abstract}

\section{Introduction}
A subset $\mathcal{S}$ of $\mathbb{N}^r$ is termed \textit{equidistributed} if for some constant $D_\mathcal{S}$ and  for any function $f\in C([0,1]^r)$,
\begin{equation}
\label{eq:def_equidis}
\lim_{n\to\infty} \ \frac{1}{n^r} \sum_{\mathbf{x}\le n, \, \mathbf{x}\in \mathcal{S}} f\Big(\frac{x_1}{n}, \dots, \frac{x_r}{n}\Big)= D_\mathcal{S}\, \int_{[0,1]^r} f(u_1,\dots, u_r)\, du_1\cdots du_r\, .
\end{equation}
For $\mathbf{x}\in \mathbb{N}^r$, the notation $\mathbf{x}\le n$ means simply that each coordinate $x_i$ of $\mathbf{x}$ is at most $n$, while, for $\mathbf{x}\in \mathbb{N}^r$ and $\boldsymbol{\beta} \in [0,\infty)^r$, we write $\mathbf{x} \le
\boldsymbol{\beta}$ if $x_j \le \beta_j$ for each $j=1, \ldots, n$.

If $\mathcal{S}$ is equidistributed, the constant $D_{\mathcal{S}}$ is the asymptotic density of $\mathcal{S}$:
\begin{equation}
\label{eq:def_density}
D_{\mathcal{S}}=\lim_{n\to\infty} \ \frac{1}{n^r} \#\{\mathbf{x}\le n, \, \mathbf{x}\in \mathcal{S}\}\, .
\end{equation}


Actually,  condition \eqref{eq:def_equidis} is equivalent to
\begin{equation}
\label{eq:def_equidis_alt}
\lim_{n\to\infty} \ \frac{1}{n^r} \#\{\mathbf{x}\le n\boldsymbol{\alpha}, \, \mathbf{x}\in \mathcal{S}\}=D_{\mathcal{S}}\, |\boldsymbol{\alpha}|\,,
\end{equation}
for each $\boldsymbol{\alpha} \in [0,1]^r$. We use  $|\boldsymbol{\alpha}|$ to denote the product $\alpha_1 \alpha_2 \cdots \alpha_r$.

For a subset $\mathcal{S}$ of $\mathbb{N}^r$, its discrepancy function $\Delta_{\mathcal{S}}: \mathbb{N}\to [0,1]$ is defined as
$$
\Delta_{\mathcal{S}}(n)=\sup_{\boldsymbol{\alpha} \in[0,1]^r}\bigg| \frac{\#\{\mathbf{x}\le n\boldsymbol{\alpha}, \, \mathbf{x}\in \mathcal{S}\}}{\#\{\mathbf{x}\le n, \, \mathbf{x}\in \mathcal{S}\}}-|\boldsymbol{\alpha}|\bigg|\, .
$$
For an equidistributed set $\mathcal{S}$, it turns out that $\Delta_{\mathcal{S}}(n)\to 0$ as $n\to\infty$; in other terms, condition \eqref{eq:def_equidis_alt} holds uniformly in $\boldsymbol{\alpha}$.

\smallskip

A classical theorem of Dirichlet asserts that the probability that two random integers are coprime is $1/\zeta(2)$, that is,
$$
\lim_{n\to\infty}\ \frac{1}{n^2} \#\big\{(i,j): 1\le i,j\le n;\ \gcd(i,j)=1\big\}=\frac{1}{\zeta(2)}
$$
(see, for instance, Theorem 332 in \cite{HW}). Indeed, the set of \textit{totient} points, those points in $\mathbb{N}^2$ whose coordinates are coprime, are equidistributed; a result which can be traced back to an observation of D.\,N. Lehmer, see Chapter~IV of~\cite{Lehmer}.

\smallskip

In this paper we will study equidistributional properties (in particular,  asymptotic density and bounds on the discrepancy function) of relevant sets  of ``coprime'' \mbox{$r$-tuples} in $\mathbb{N}^r$, where, for $r \ge 3$,  ``coprime'' could mean just ``mutually coprime'' or, more demandingly, ``pairwise coprime''.  These results  shall prove useful elsewhere (see~\cite{FF1}). It should be pointed out that because of the nature of the sets $\mathcal{S}$ under study in this paper, a variation of the proof of \eqref{eq:def_density} usually provides a proof of \eqref{eq:def_equidis_alt}.

For the whole set $\mathcal{S}$ of points in $\mathbb{N}^r$ with mutually coprime or pairwise coprime coordinates, these results are (essentially) known.
For instance, if $\mathcal{S}$ is the set of mutually coprime $r$-tuples, then $D_{\mathcal{S}}=1/\zeta(r)$, while for the set of pairwise coprime $r$-tuples, $D_{\mathcal{S}}=T_r$, where the constant $T_r$ is defined in~\eqref{eq:proportion of PC}; see Section \ref{sub:prelims}.

\smallskip

As a sample of our results,  consider the following special totient points: Fix a reference $r$-tuple $\mathbf{a}=(a_1, \dots, a_r)$ such that the coordinates $a_j$ are pairwise coprime. The set $\mathcal{S}$ of interest comprises the $r$-tuples $\mathbf{x}=(x_1,\dots, x_r)$ of integers with mutually coprime coordinates and such that each $x_i$ is a multiple of the corresponding $a_i$, for $i=1,\dots, r$. As we will se in~Section \ref{subsection:coprimality plus divisibility}, $\mathcal{S}$ is equidistributed with constant
\begin{align*}
\frac{1}{\zeta(r)} \ \frac{\varphi_{r-1}(|\mathbf{a}|)}{\varphi_{r}(|\mathbf{a}|)}
\end{align*}
(see the definition of the Jordan totient function $\varphi_r$ in~\eqref{eq:Jordan function}). In the case $r=2$, for $\mathbf{a}=(a,b)$, with $\gcd(a,b)=1$, one would perhaps na\"{\i}vely expect the proportion of totient points on the lattice $a\mathbb{N}\oplus b\mathbb{N}$ to be $\frac{1}{\zeta(2)}\frac{1}{ab}$; while, in fact,  the above expression reduces to
$$
\frac{1}{\zeta(2)}\, \frac{1}{ab}\, \bigg[\frac{1}{\prod_{p|ab} (1+1/p)}\bigg]\,.
$$
Notice the correction factor within the brackets. This result was already obtained, for the case $\mathbf{a}=(a,1)$, by D.\,N. Lehmer in 1900 (see Theorem I, Chapter IV, in~\cite{Lehmer})\footnote{He applied this estimate to the problem of counting the number of right triangles whose sides are mutually coprime integers, and such that the hypotenuse is less or equal than $N$. The problem was revisited by his son D.\,H. Lehmer in \cite{Lehmer-son}. By the way, D.\,N. Lehmer also consider some  mixed divisibility/coprimality conditions.}.

\subsection{Preliminaries, notation and background results}\label{sub:prelims} For $r$-tuples of integers, there are several notions of ``coprimality''. The integers $a_1,\dots, a_r$ are \textit{mutually coprime} if $\gcd(a_1,\dots,a_r)=1$; and they are \textit{pairwise coprime} if $\gcd(a_i, a_j)=1$ for each $i\ne j$, $1\le i,j\le r$. We will refer to this by simply writing $(a_1,\dots,a_r)\in \mathcal{C}$ and $(a_1,\dots, a_r)\in \mathcal{PC}$, respectively. For $\mathbf{a}=(a_1, \ldots, a_r)$, we abbreviate $\gcd(a_1,\dots,a_r)=\gcd(\mathbf{a})$.

 Another notion of coprimality, intermediate between mutual and pairwise coprimality, is the following: for fixed $2\le k\le r$, we will say that the integers $(a_1,\dots, a_r)$ are \textit{$k$-wise relatively prime} (or simply $k$-coprime, or $k\mathcal{C}$) if \textit{any}~$k$ of them are relatively prime. That is, if $\gcd(a_{i_1},\dots, a_{i_k})=1$ for any set of $k$ indexes $1\le i_1<\dots< i_k\le r$, or alternatively, if each prime $p$ divides at most $k-1$ of them. The case $k=2$ is pairwise coprimality, while $k=r$ corresponds to mutual coprimality.

\medskip

Throughout, we will use the following probabilistic setting: for any given integer $n \ge 2$, denote by $X^{(n)}_1,
X^{(n)}_2, \ldots$ a sequence of independent random variables
uniformly distributed in $\{1, 2, \ldots, n\}$  and defined in a
certain given probability space endowed with a probability~$\P$.

Fix $r\ge 2$. Concerning mutual coprimality, we have
\begin{equation}
\label{eq:proportion of C}
\lim_{n \to \infty}\P\big(\big(X^{(n)}_1, \ldots, X^{(n)}_r\big) \in\mathcal{C}\big)=\lim_{n \to \infty}\P\big(\gcd\big(X^{(n)}_1, \ldots, X^{(n)}_r\big) =1 \big)=\frac{1}{\zeta(r)}\, ,
\end{equation}
that is, the asymptotic proportion of mutually coprime $r$-tuples of integers is $1/\zeta(r)$. The case $r=2$ is Dirichlet's result. The extension to $r>2$ can be traced back to E.~Ces\`{a}ro (\cite{Ce3}, page~293); see also, for instance, \cite{Ch}, \cite{HS} and \cite{Ny}.

\smallskip
For pairwise coprimality, we have
\begin{equation}\label{eq:proportion of PC}
\lim_{n \to \infty}\P\big(\big(X^{(n)}_1, \ldots, X^{(n)}_r\big) \in \mathcal{PC}\big)
=\prod_{p} \Big(\Big(1-\frac{1}{p}\Big)^{r}+\frac{r}{p}\Big(1-\frac{1}{p}\Big)^{r-1}\Big):=T_r.
\end{equation}
This fact was proved by L. Toth (\cite{To2004}) and by J.~Cai and E.~Bach (\cite{CB2001}).
For $r=2$, mutual and pairwise coprimality coincide ($T_2={1}/{\zeta(2)}$).

Recently, J. Hu (see Corollary 2 in \cite{Hu2}) has proved that
\begin{equation}\label{eq:proportion of kC}
\lim_{n \to \infty}\P\big(\big(X^{(n)}_1, \ldots, X^{(n)}_r\big) \in k\mathcal{C}\big)
=\prod_{p} \P(\textsc{bin}(r,1/p)\le k-1)
\end{equation}
using a recursive scheme related to the one in \cite{To2004}. See Section \ref{subsection:primality results} for an alternative proof.

\medskip

For each $r\ge 1$, the \textit{$r$-Jordan totient function} is given by
\begin{equation}
\label{eq:Jordan function}
\varphi_r(1)=1, \quad \varphi_r(a)=a^r \, \prod_{p |a} \Big(1-\frac{1}{p^r}\Big)\quad\text{for $a>1$}\, ;
\end{equation}
notice  that $\varphi_1(a)=\varphi(a)$,  Euler's $\varphi$ function. For each integer~$s$, the
function $\Psi_s$ is defined~as
\begin{equation}
\label{eq:Psi function}
\Psi_s(1)=1, \quad \Psi_s(a)=a\, \prod_{p |a} \Big(1+\frac{s}{p}\Big) \quad\text{for $a>1$}. 
\end{equation}
Observe  that $\Psi_0$ is the identity function, while $\Psi_{-1}\equiv \varphi$. In this paper, we just need $s\ge 1$, the case $s=1$ being the Dedekind Psi function. Observe that, for prime $p$ and positive integer $n$, $\Psi_s(p^n)=p^n\,(1+{s}/{p}).$

We shall also use the following fact: for any arithmetical function $F$,
\begin{equation}
\label{eq:use of mu}
\sum_{\mathbf{x}\le n, \,\gcd(\mathbf{x})=1} F(\mathbf{x})=\sum_{k=1}^n \mu(k) \sum_{\mathbf{x}\le n, \,k|\mathbf{x}} F(\mathbf{x})
=\sum_{k=1}^n \mu(k) \sum_{\mathbf{y}\le n/k} F(k\mathbf{y}),
\end{equation}
thanks to the properties of the M\"{o}bius function. Notice that on the right-hand side no coprimality restriction appears.

Finally, we shall denote the ordered sequence of primes by $p_1,p_2,\dots$

\subsubsection*{Organization of the paper}
Section \ref{section:proportion for coprimes} revisits  the proofs of the basic results about the asymptotic proportion of coprime $r$-tuples. In Section \ref{section:coprimality plus}, we will study equidistributional properties of several variations on coprimality with some extra conditions. Some  bounds for discrepancy functions will be obtained in Section~\ref{section:equidistribution for C and PC}.

\section{Asymptotic density of coprime $r$-tuples}\label{section:proportion for coprimes}

In this section we shall revisit some known results about  the asymptotic density of totient points following the approach in Cai-Bach, \cite{CB2001}. We shall recast and formalize its ingredients a bit, so that it could be applied in other contexts of interest.

\subsection{The Cai--Bach approach}\label{subsection:atomic lemma}
Fix integers $N\ge 1$ and $r\ge 2$.
Consider a matrix $M$ of dimensions $N\times r$, with entries $m_{ij}\in\{0,1\}$: the rows of $M$ are labeled with the primes $p_1,\dots, p_N$.


\smallskip
Given a random sample $\mathbf{X}^{(n)}=(X_1^{(n)},\dots,X_r^{(n)})$ of length~$r$, denote by $\mathbf{M}^{(n)}$ the associated $N\times r$ (random) matrix encoding the divisibility properties of the sample: the entry $(i,j)$ of the matrix $\mathbf{M}^{(n)}$ will be 1 if the prime $p_i$ divides the coordinate~$X_j^{(n)}$, and 0 otherwise.


\smallskip

For $i=1,\dots, N$ and $j=1,\dots, r$, denote by $I_{ij}$ a collection of independent Bernoulli random variables with success probability  $1/p_i$. Let $\mathbf{I}$ denote the $N \times r$ matrix whose entries are the $I_{ij}$. Observe that
$$
\P\big(\mathbf{I}=M\big)=\prod_{i=1}^N \prod_{j=1}^r \P(I_{ij}=m_{ij})\, .
$$

\begin{lemma}
\label{lemma:Cai-Bach atomizacion}Fix $N\ge 1$ and $r\ge 2$, and a matrix $M=(m_{ij})$ of dimensions $N\times r$, with entries $0$ or $1$. Then,
\begin{equation} \label{eq:Cai-Bach atomizacion}
\lim_{n\to\infty}\ \P\big(\mathbf{M}^{(n)}=M\big)=\P\big(\mathbf{I}=M\big)
\end{equation}
\end{lemma}
This lemma is just a formulation of the asymptotic independence of divisibility by different primes.
\begin{proof}[Proof of Lemma {\upshape\ref{lemma:Cai-Bach atomizacion}}]
We just analyze the first coordinate, $X_1^{(n)}$. Say that in the first column of $M$ there are $a$ ones (corresponding to the rows indexed with primes $s_1,\dots, s_a$) and $b$ zeros (corresponding to the primes $t_1,\dots, t_b$). Write $\widetilde{s}=s_1\cdots s_a$ (or $\widetilde{s}=1$ if $a=0$).

\smallskip

For $p$ prime, consider the set
$$
H_p=\{1\le k\le n : p| k\},
$$
and let $\overline{H}_p$ be its complement in $\{1,\dots,n\}$. Observe that
$
|H_p|=\lfloor n/p\rfloor
$ and $|\overline{H}_p|=n-\lfloor n/p\rfloor$. For primes $p,q$, using that $\big\lfloor \frac{\lfloor x\rfloor}{n}\big\rfloor=
\big\lfloor \frac{x}{n}\big\rfloor$ for $x\ge 0$ and $n\in\mathbb{N}$, one readily checks that
$$
|H_p\cap H_q|=\Big\lfloor \frac{n}{pq}\Big\rfloor,\quad
|H_p\cap \overline{H}_q|=\Big\lfloor \frac{n}{p}\Big\rfloor-\Big\lfloor \frac{n}{pq}\Big\rfloor,\quad
|\overline{H}_p\cap \overline{H}_q|=n-\Big\lfloor \frac{n}{p}\Big\rfloor-\Big\lfloor \frac{n}{q}\Big\rfloor+\Big\lfloor \frac{n}{pq}\Big\rfloor.
$$
In general, in  our case,
$$
\big|H_{s_1}\cap\cdots \cap H_{s_a}\cap \overline{H}_{t_1}\cap \cdots \cap \overline{H}_{t_b}\big|=
\Big\lfloor \frac{n}{\widetilde{s}}\Big\rfloor
-\sum_{i=1}^b \Big\lfloor \frac{n}{\widetilde{s}\, t_i}\Big\rfloor
+\sum_{i<j} \Big\lfloor \frac{n}{\widetilde{s}\, t_i\, t_j}\Big\rfloor- \cdots
$$
This means that
\begin{align*}
\P\big(s_1, \dots, s_a\,|\, X_1^{(n)}, &\text{ and } t_1,\dots, t_b\, \nmid\, X_1^{(n)}\big)=\frac{1}{n}\Big(
\Big\lfloor \frac{n}{\widetilde{s}} \Big\rfloor
-\sum_{i=1}^b \Big\lfloor \frac{n}{\widetilde{s}\, t_i}\Big\rfloor
+\sum_{i<j} \Big\lfloor \frac{n}{\widetilde{s}\, t_i\, t_j}\Big\rfloor- \cdots\Big)
\\
&\xrightarrow{n\to\infty}\
\frac{1}{\widetilde{s}}\
\Big(1
-\sum_{i=1}^b \frac{1}{t_i}
+\sum_{i<j} \frac{1}{t_i\, t_j}- \cdots\Big)=\frac{1}{\widetilde{s}}\ \prod_{j=1}^b \Big(1-\frac{1}{t_j}\Big)
\\
&=\prod_{i=1}^a \frac{1}{s_i}\ \prod_{j=1}^b \Big(1-\frac{1}{t_j}\Big)=\prod_{i=1}^N \P(I_{i1}=\varepsilon_{i1}).
\end{align*}
The independence of the coordinates $X_j^{(n)}$ gives the result.\end{proof}

\begin{corollary}[of Proof]With the hypothesis of Lemma {\upshape\ref{lemma:Cai-Bach atomizacion}}, if ${\boldsymbol\alpha}\in [0,1]^r$, then
\begin{equation} \label{eq:Cai-Bach atomizacion alpha}
\lim_{n\to\infty}\ \P\big(\mathbf{X}^{(n)}\le n{\boldsymbol\alpha}\,,\,\mathbf{M}^{(n)}=M\big)=|\boldsymbol\alpha|\,\P\big(\mathbf{I}=M\big)
\end{equation}
\end{corollary}
\begin{proof}
Follow the same proof but, for $0\le \alpha\le 1$, for prime $p$ define $H_p=\{1\le k\le n\alpha : p| k\}$ and
$\overline{H}_p=\{1\le k\le n\alpha : p\nmid k\}$.
\end{proof}

Some particular cases of Lemma \ref{lemma:Cai-Bach atomizacion} are in order.

\smallskip
a) Fix sets $A_1,\dots,A_N\subset \{1,\dots,r\}$ of sizes $h_1,\dots, h_N$, respectively, and consider the matrix~$M$ with $m_{ij}=1$ if $j\in A_i$ (and 0 otherwise). Then
\begin{align}
\nonumber
\P\big(\mathbf{X}^{(n)}\colon  p_i| X_j^{(n)} &\,\text{ if } j\in A_i\text{ and } p_i\nmid  X_j^{(n)} \text{ if } j\notin A_i\big)
\\ \label{eq:atomiza para Aes}
&=
\P\big(\mathbf{M}^{(n)}=M\big)\longrightarrow
\prod_{i=1}^N \frac{1}{p_i^{h_i}}\, \Big(1-\frac{1}{p_i}\Big)^{r-h_i}\quad\text{as $n\to\infty$.}
\end{align}

b)  For $h_1,\dots, h_N$ fixed, consider the collection of matrices $\mathcal{M}$ with exactly $h_i$ ones in each row $i$, that is, $\mathcal{M}(h_1,\dots, h_N)=\{M: \sum_{j=1}^r m_{ij}=h_i, \text{ for }  i=1,\dots, N\}$.
Notice that there are $\prod_{i=1}^N \binom{r}{h_i}$ different matrices in $\mathcal{M}$.
Then, using \eqref{eq:atomiza para Aes}, and observing that $\{\mathbf{M}=M_1\}$ and $\{\mathbf{M}=M_2\}$ are disjoint for $M_1 \neq M_2$, we deduce
\begin{align}
\nonumber
\P\big(\mathbf{X}^{(n)}\colon &\, p_i \text{ divides exactly $h_i$ of the $X_j^{(n)}$}\big)=\P(\mathbf{M}^{(n)}\in \mathcal{M})=\sum_{M\in\mathcal{M}}\P(\mathbf{M}^{(n)}=M)
\\ \nonumber
&\longrightarrow \sum_{M\in\mathcal{M}}\prod_{i=1}^N \frac{1}{p_i^{h_i}}\, \Big(1-\frac{1}{p_i}\Big)^{r-h_i}
=\prod_{i=1}^N {r\choose {h_i}}\, \frac{1}{p_i^{h_i}}\, \Big(1-\frac{1}{p_i}\Big)^{r-h_i}
\\ \label{eq:atomiza para h}
&=\prod_{i=1}^N \P(\textsc{bin}(r,1/p_i)=h_i).
\end{align}

c) Finally,
\begin{lemma}\label{lemma:atomiza para <=h}
For $h_1,\dots, h_N$ fixed,
\begin{equation}
\label{eq:atomiza para <=h}
\lim_{n\to\infty}\ \P\big(\mathbf{X}^{(n)}\colon  p_i \text{ divides at most $h_i$ of the coordinates of $\mathbf{X}^{(n)}$}\big)
=\prod_{i=1}^N \P(\textsc{bin}(r,1/p_i)\le h_i).
\end{equation}
More generally, for $\boldsymbol\alpha\in[0,1]^r$,
\begin{align}
\label{eq:atomiza para <=h alpha}
\lim_{n\to\infty}\ \P\big(\mathbf{X}^{(n)}\le n\boldsymbol\alpha\colon  p_i &\,\text{ divides at most $h_i$ of the coordinates of $\mathbf{X}^{(n)}$}\big)
\\\nonumber
&=|\boldsymbol\alpha|\, \prod_{i=1}^N \P(\textsc{bin}(r,1/p_i)\le h_i).
\end{align}

\end{lemma}
\begin{proof}
Just the argument for the case \eqref{eq:atomiza para <=h}. Consider the collection of matrices
$$
\mathcal{M}'(h_1,\dots, h_N)=\big\{M: {\textstyle\sum_{j=1}^r} m_{ij}\le h_i \text{ for } i=1,\dots, N\big\}=
\bigcup_{0\le k_i\le h_i} \mathcal{M}(k_1,\dots, k_N).
$$
Then, thanks to \eqref{eq:atomiza para h}, and observing that the above is a disjoint union,  we obtain
\begin{align*}
\P\big(\mathbf{X}^{(n)}\colon &\, p_i \text{ divides at most $h_i$ of the $X_j^{(n)}$}\big)=\P(\mathbf{M}^{(n)}\in \mathcal{M}')
\\
&=\sum_{0\le k_i\le h_i}\P\big(\mathbf{M}^{(n)}\in\mathcal{M}(k_1,\dots, k_N)\big)
\longrightarrow\sum_{0\le k_i\le h_i}\prod_{i=1}^N \P(\textsc{bin}(r,1/p_i)=k_i)
\\
&=\prod_{i=1}^N \P(\textsc{bin}(r,1/p_i)\le h_i).\qedhere
\end{align*}\end{proof}

\subsection{Proof of coprimality results}\label{subsection:primality results}
In this section we will prove the result \eqref{eq:proportion of kC} on the proportion of $r$-tuples that are $k$-wise coprime (obtaining \eqref{eq:proportion of C} and \eqref{eq:proportion of PC} as particular cases). We will follow the approach used by  Cai and Bach~\cite{CB2001} for the case of pairwise coprimality. See \cite{To2004} and \cite{Hu2} for alternative approaches.
The length of the sample, $r\ge 2$, will be fixed henceforth.

\smallskip

Fix $2\le k\le r$. For each prime $p$, define
$$
G_p^{(n)} =\{\mathbf{x}\le n: \text{$p$ divides at most $k-1$ of the $x_j$}\}.
$$

Observe that
$$
k\mathcal{C}^{(n)}:=\{\mathbf{x}\le n: \mathbf{x}\in k\mathcal{C}\}=\bigcap_p \,G_p^{(n)}.
$$
Using \eqref{eq:atomiza para <=h}, we get
$$
\frac{1}{n^r}\, \big|k\mathcal{C}^{(n)}\big|=\frac{1}{n^r}\, \Big|\bigcap_{j=1}^\infty \,G_{p_j}^{(n)}\Big|\le \frac{1}{n^r}\, \Big|\bigcap_{j=1}^N \,G_{p_j}^{(n)}\Big|\xrightarrow{n\to\infty} \prod_{j=1}^N \P(\textsc{bin}(r,1/p_j)\le k-1),
$$
so that
$$
\limsup_{n\to\infty}\frac{1}{n^r}\, \big|k\mathcal{C}^{(n)}\big|\le \prod_{j=1}^N \P(\textsc{bin}(r,1/p_j)\le k-1).
$$
Finally, as $N$ is arbitrary,
$$
\limsup_{n\to\infty}\frac{1}{n^r}\, \big|k\mathcal{C}^{(n)}\big|\le \prod_p \P(\textsc{bin}(r,1/p)\le k-1).
$$

Now, for fixed $N$, observe that
$$
\bigcap_{j=1}^N \, G_{p_j}^{(n)}\setminus \bigcap_{j=1}^\infty \, G_{p_j}^{(n)} \subset \bigcup_{j=N+1}^\infty\, B_{p_j}^{(n)}\,,
$$
where $B_{p}^{(n)}$ is the complementary of $G_{p}^{(n)}$ in $\{1,\dots,n\}^r$, that is, the set of $\mathbf{x} \le n$  such that~$p$ divides $k$ (or more) of the $x_j$. This means that
$$
\Big|\bigcap_{j=1}^N \, G_{p_j}^{(n)}\Big|-|k\mathcal{C}_r^{(n)}| \le \sum_{j=N+1}^\infty\, |B_{p_j}^{(n)}|\,.
$$
Now,
$$
B_p^{(n)} 
\subset \bigcup_{1\le i_1<\cdots <i_k\le r} \{\mathbf{x}\le n: \text{$p$ divides $x_{i_1},\dots, x_{i_k}$}\}.
$$
so that
$$
|B_p^{(n)}|\le {r\choose k} \#\{\mathbf{x}\le n : \text{$p$ divides $x_1,\dots, x_k$}\}={r\choose k} \, \Big\lfloor\frac{n}{p}\Big\rfloor^k\, n^{r-k}\le {r\choose k} \frac{n^r}{p^k}.
$$
Then, the tail is bounded by
$$
\sum_{j=N+1}^\infty\, |B_{p_j}^{(n)}|
\le {r\choose k} \, n^r\sum_{j=N+1}^\infty \frac{1}{p_j^k}\,.
$$
This yields
$$
\frac{1}{n^r}\, |k\mathcal{C}^{(n)}|\ge \frac{1}{n^r}
\Big|\bigcap_{j=1}^N \, G_{p_j}^{(n)}\Big|- {r\choose k} \, \sum_{j=N+1}^\infty \frac{1}{p_j^k}\,,
$$
and so, using again \eqref{eq:atomiza para <=h},
$$
\liminf_{n\to\infty}\, \frac{1}{n^r}\, |k\mathcal{C}^{(n)}|\ge \Big[\prod_{j=1}^N \P\big(\textsc{bin}(r,1/p_j)\le k-1\big)\Big]- {r\choose k} \, \sum_{j=N+1}^\infty \frac{1}{p_j^k}\,.
$$
We finish the proof by letting $N\to\infty$.\qed

\medskip

This proof gives directly an equidistributional result.
\begin{proposition}
The set of $k\mathcal{C}$-points in $\mathbb{N}^r$ is equidistributed.
\end{proposition}
\begin{proof}
For fixed $\boldsymbol\alpha\in[0,1]^r$ set
$$
G_p^{(n)} =\{\mathbf{x}\le n\boldsymbol\alpha: \text{$p$ divides at most $k-1$ of the $x_j$}\}
$$
and
$$
B_p^{(n)} =\{\mathbf{x}\le n\boldsymbol\alpha: \text{$p$ divides at least $k$ of the $x_j$}\},
$$
and proceed exactly as above to obtain
$$
\lim_{n\to\infty}\frac{1}{n^r}\,\{\mathbf{x}\le n\boldsymbol\alpha: \mathbf{x}\in k\mathcal{C}\}=|\boldsymbol\alpha|\, \prod_p \P(\textsc{bin}(r,1/p)\le k-1)\,,
$$
as desired.\end{proof}

\section{Special totient points}\label{section:coprimality plus}
\subsection{Coprimality with extra coprimality conditions}\label{subsection:coprimality plus coprimes}
Let $\mathbf{a}\in \mathcal{PC}$. We are interested in estimating the proportion of $r$-tuples $\mathbf{x}$ of integers that are (mutually or pairwise) coprime and such that, additionally, each coordinate $x_j$ is coprime with the corresponding $a_j$.

\smallskip
We introduce some notation. We say that an $r$-tuple of integers $\mathbf{x}=(x_1,\dots, x_r)$ belongs to $\mathcal{PC}_\mathbf{a}$ if $\mathbf{x}\in \mathcal{PC}$ and, additionally, $\gcd(a_i, x_i)=1$ for all $i=1,\dots, r$ (abbreviated, $\mathbf{x}\perp\mathbf{a}$). Analogously, we say that $\mathbf{x}=(x_1,\dots, x_r)\in \mathcal{C}_\mathbf{a}$ if $\mathbf{x}\in \mathcal{C}$ and $\mathbf{x}\perp\mathbf{a}$.


\begin{theorem}\label{thm:C and PC with extra coprimality}
Given $\mathbf{a}\in \mathcal{PC}$, we have:

\smallskip
{\rm a)} For pairwise coprimality,
\begin{equation}
\label{eq:PC with extra coprimality}
\lim_{n \to \infty}\P\big(\mathbf{X}^{(n)} \in\mathcal{PC}_\mathbf{a}\big)=T_r\, \frac{\Psi_{r-2}(|\mathbf{a}|)}{\Psi_{r-1}(|\mathbf{a}|)}.
\end{equation}
The function $\Psi_s$ was defined in \eqref{eq:Psi function}.

\smallskip
{\rm b)} For mutual coprimality,
\begin{equation}
\label{eq:C with extra coprimality}
\lim_{n \to \infty}\P\big(\mathbf{X}^{(n)} \in\mathcal{C}_\mathbf{a}\big)=\frac{1}{\zeta(r)}\, \frac{\varphi(|\mathbf{a}|)}{\varphi_{r}(|\mathbf{a}|)}\, |\mathbf{a}|^{r-1},
\end{equation}
where the Jordan function $\varphi_r$ was defined in \eqref{eq:Jordan function}.
\end{theorem}
\begin{proof} Denote by $\mathcal{P}_1,\dots, \mathcal{P}_r$ the (disjoint) sets of primes dividing $a_1,\dots, a_r$, respectively (if some $a_i=1$, then we set $\mathcal{P}_i=\emptyset$). Write $\mathcal{P}=\cup_{j=1}^r \mathcal{P}_j$.

\smallskip

a) Fix $N$ large enough so that $\mathcal{P}\subset\mathcal{P}_N:=\{p_1,\dots, p_N\}$. Following the approach of Section \ref{subsection:atomic lemma}, we consider the matrices $M$ of dimensions $N\times r$  which fulfil the following requirements:
\begin{itemize}\itemsep=0pt
\item in each row there is at most one 1 (to ensure pairwise coprimality);
\item in a row labelled with a prime $p\in\mathcal{P}_i$ there is 0 in column $i$ (to ensure that $p\nmid X_i^{(n)})$.
\end{itemize}
As in Lemma \ref{lemma:atomiza para <=h}, we deduce
\begin{align*}
\limsup_{n\to\infty}\P\big(&\,\mathbf{X}^{(n)}\in\mathcal{PC}_\mathbf{a}\big)
\\
&\le \prod_{p\in\mathcal{P}_N\setminus \mathcal{P}} \!\!\!\P(\textsc{bin}(r,\tfrac{1}{p})\le 1)\cdot \prod_{j=1}^r \prod_{p\in \mathcal{P}_j} \Big(1-\frac{1}{p}\Big)\P(\textsc{bin}(r-1,\tfrac{1}{p})\le 1)
\\
&=\prod_{p\in\mathcal{P}_N} \P(\textsc{bin}(r,\tfrac{1}{p})\le 1)\cdot \prod_{j=1}^r \prod_{p\in \mathcal{P}_j} \Big(1-\frac{1}{p}\Big)\frac{\P(\textsc{bin}(r-1,\tfrac{1}{p})\le 1)}{\P(\textsc{bin}(r,\tfrac{1}{p})\le 1)}.
\end{align*}
Straightforward manipulations yield
$$
\prod_{p\in \mathcal{P}_1} \Big(1-\frac{1}{p}\Big)\frac{\P(\textsc{bin}(r-1,{1}/{p})\le 1)}{\P(\textsc{bin}(r,{1}/{p})\le 1)}=\prod_{p\in \mathcal{P}_1} \frac{1+(r-2)/{p}}{1+(r-1)/{p}}=\prod_{p\in\mathcal{P}_1} \frac{\Psi_{r-2}(p)}{\Psi_{r-1}(p)}=\frac{\Psi_{r-2}(a_1)}{\Psi_{r-1}(a_1)}\,.
$$
Now, as $\Psi$ is multiplicative, and $N$ is arbitrary, we get
$$
\limsup_{n\to\infty}\ \P\big(\mathbf{X}^{(n)}\in\mathcal{PC}_\mathbf{a}\big)\le T_r\, \frac{\Psi_{r-2}(|\mathbf{a}|)}{\Psi_{r-1}(|\mathbf{a}|)}.
$$

On the other direction, as $\mathcal{P}_N$ includes all the primes in $\mathcal{P}$ for $N$ large enough, the same argument used for the case of pairwise coprimality (with no extra conditions, see Section~\ref{subsection:primality results}) finishes the proof.

\smallskip
b) In this case, the matrices $M$ have at most $r-1$ ones in each row (to ensure mutual coprimality) and, again,
a 0 in the $i$-th column if the prime labeling the row belongs to $\mathcal{P}_i$. The product of probabilities to be considered is now
\begin{align*}
\prod_{p\in\mathcal{P}_N\setminus \mathcal{P}} &\,\P(\textsc{bin}(r,\tfrac{1}{p})\le r-1)\cdot \prod_{j=1}^r \prod_{p\in \mathcal{P}_j} \Big(1-\frac{1}{p}\Big)\P(\textsc{bin}(r-1,\tfrac{1}{p})\le r-1)
\\
& =\prod_{p\in\mathcal{P}_N} \Big(1-\frac{1}{p^r}\Big)\cdot \prod_{j=1}^r \prod_{p\in \mathcal{P}_j} \frac{(1-{1}/{p})}{(1-{1}/{p^r})}
=\frac{\varphi(|\mathbf{a}|)}{\varphi_{r}(|\mathbf{a}|)}\, (|\mathbf{a}|)^{r-1}\ \prod_{p\in\mathcal{P}_N} \Big(1-\frac{1}{p^r}\Big).
\end{align*}
The argument for the tail is analogous to that used in the proof of the case of mutual coprimality (with no extra conditions, see Section~\ref{subsection:primality results}).
\end{proof}

The equidistributional version of Theorem \ref{thm:C and PC with extra coprimality}, part a), reads as follows. It  will be useful elsewhere (see \cite{FF1}).

\begin{corollary}
\label{coro:PCE equidistributed} Fix $r\ge 2$. For $\mathbf{a}=(a_1,\dots, a_r)\in\mathcal{PC}$, the set $\mathcal{PC}_\mathbf{a}\in \mathbb{N}^r$ is equidistributed with constant
$$
T_r\, \frac{\Psi_{r-2}(|\mathbf{a}|)}{\Psi_{r-1}(|\mathbf{a}|)}.
$$
\end{corollary}

%

\smallskip

There is a corresponding version of Theorem \ref{eq:PC with extra coprimality} for  $k\mathcal{C}$-coprimality, in general, but the expressions of the asymptotic densities are a bit too cumbersome.

Next is a more general version of Theorem \ref{eq:PC with extra coprimality}, in which repeated $a_j$ are allowed.
\begin{theorem}\label{thm:C and PC with extra coprimality, repeated}
For $m\le r$, fix a $m$-tuple $\mathbf{a}\in \mathcal{PC}$. Let $\cup_{i=1}^k \mathcal{A}_i$ be a partition of $\{1,\dots,r\}$; write $b_i=|\mathcal{A}_i|$. Then,

\smallskip
{\rm a)} The asymptotic proportion of $r$-tuples $\mathbf{x}$ of integers  such that $\mathbf{x}\in \mathcal{PC}$ and, for $j=1,\dots,k$, $\gcd(a_j, x_i)=1$ if $k\in \mathcal{A}_i$, is given by
\begin{equation}
\label{eq:PC with extra coprimality, repeated}
T_r\, \frac{1}{\Psi_{r-1}(|\mathbf{a}|)}\prod_{i=1}^m {\Psi_{r-b_i-1}(a_i)}.
\end{equation}

{\rm b)} The asymptotic proportion of $r$-tuples $\mathbf{x}$ of integers  such that $\mathbf{x}\in \mathcal{C}$ and, for $j=1,\dots,k$, $\gcd(a_j, x_i)=1$ if $k\in \mathcal{A}_i$, is given by
\begin{equation}
\label{eq:C with extra coprimality, repeated}
\frac{1}{\zeta(r)}\, \frac{|\mathbf{a}|^r}{\varphi_{r}(|\mathbf{a}|)}\prod_{i=1}^m \Big(\frac{\varphi(a_j)}{a_j}\Big)^{b_i}.
\end{equation}
\end{theorem}

The case $m=r$ of Theorem \ref{thm:C and PC with extra coprimality, repeated} is Theorem \ref{thm:C and PC with extra coprimality}. The case $m=1$ corresponds to the case of pairwise (or mutual) coprime integers that are, additionally, prime with a fixed~$a$. Equation~\eqref{eq:PC with extra coprimality, repeated} reads, in this case,
$$
T_r\, \frac{\varphi(a)}{\Psi_{r-1}(a)}=T_r\, \prod_{p|a} \frac{1-1/p}{1-1/(r-1)/p}=
T_r\, \prod_{p|a} \Big(1-\frac{r}{p+r-1}\Big),
$$
as in Toth's \cite{To2004}, page 14.

\subsection{Coprimality with extra divisibility conditions}\label{subsection:coprimality plus divisibility}

We are now interested in estimating the proportion of $r$-tuples of integers that are pairwise (or mutually) coprime when restricting to the \textit{multiples} of certain fixed numbers, that is, to the lattice $a_1\mathbb{N}\oplus\cdots \oplus a_r\mathbb{N}$.

\smallskip
We say that an $r$-tuple of integers $\mathbf{x}=(x_1,\dots, x_r)$ belongs to $\mathcal{PC}^\mathbf{a}$ if $\mathbf{x}\in \mathcal{PC}$ and, additionally, $a_i| x_i$ for all $i=1,\dots, r$ (abbreviated, $\mathbf{a}|\mathbf{x}$). Analogously, we will say that $\mathbf{x}=(x_1,\dots, x_r)\in \mathcal{C}^\mathbf{a}$ if $\mathbf{x}\in \mathcal{C}$ and $\mathbf{a}|\mathbf{x}$.

\begin{theorem}\label{thm:C and PC with extra divisibility}
Given $\mathbf{a}\in \mathcal{PC}$, we have:

\smallskip
{\rm a)} For pairwise coprimality,
\begin{equation}
\label{eq:PC with extra divisibility}
\lim_{n \to \infty}\P\big( \mathbf{X}^{(n)}\in\mathcal{PC}^\mathbf{a}\big)=T_r\, \frac{1}{\Psi_{r-1}(|\mathbf{a}|)},
\end{equation}
The function $\Psi_s$ was defined in \eqref{eq:Psi function}.

\smallskip
{\rm b)} For mutual coprimality,
\begin{equation}
\label{eq:C with extra divisibility}
\lim_{n \to \infty}\P\big(\mathbf{X}^{(n)} \in\mathcal{C}^\mathbf{a}\big)=\frac{1}{\zeta(r)}\, \frac{\varphi_{r-1}(|\mathbf{a}|)}{\varphi_{r}(|\mathbf{a}|)},
\end{equation}
where the Jordan function $\varphi_r$ was defined in \eqref{eq:Jordan function}.
\end{theorem}

\begin{proof}
a) We follow the notation of the previous section: for $\mathbf{a}\in \mathcal{PC}$, denote by $\mathcal{P}_1,\dots, \mathcal{P}_r$ the (disjoint) sets of primes dividing $a_1,\dots, a_r$, respectively (if some $a_i=1$, then we set $\mathcal{P}_i=\emptyset$). Write $\mathcal{P}=\cup_{j=1}^r \mathcal{P}_j$.

We need now to keep track of the exponent $\alpha$ of each prime $p$ appearing in the decomposition of the $a_j$.

Fix $N$ large enough so that $\mathcal{P}\subset\mathcal{P}_N:=\{p_1,\dots, p_N\}$.

Recall that we want to estimate the probability that $\big(X^{(n)}_1, \ldots, X^{(n)}_r\big) \in\mathcal{PC}$ and, additionally,  $a_i| X_i^{(n)}$ for all $i=1,\dots, r$.

Say that $p$ (with exponent $\alpha$) belongs to $\mathcal{P}_1$. We need that $p^\alpha|X_1^{(n)}$ (for the extra divisibility condition) and that $p$ divides, at most, one of the $X_j^{(n)}$ (for pairwise coprimality). This simplifies to $p^\alpha|X_1^{(n)}$ and $p\nmid X_j^{(n)}$ for $j\ne 1$. The random variable registering this situation may be written as
$$
I_{p^\alpha}(X_1^{(n)})\prod_{i\ne 1} (1-I_p(X_i^{(n)})),
$$
where $I_a(X_j)=1$ if $a|X_j$ and 0 otherwise. Putting all the primes together, we have to consider the random variable
$$
\prod_{p\in \mathcal{P}_N\setminus \mathcal{P}} \uno_{\{\sum_{i=1}^r I_p(X_i^{(n)})\le 1\}}\cdot  \prod_{j=1}^r \prod_{p\in\mathcal{P}_j} I_{p^\alpha}(X_j^{(n)})\prod_{i\ne j} (1-I_p(X_i^{(n)}))
$$
(the first product ensures pairwise coprimality for the primes not belonging to the $\mathcal{P}_j$). Now, adapting Lemma~\ref{lemma:Cai-Bach atomizacion} to this situation, we get that
the product of probabilities to be considered is
\begin{align*}
\prod_{p\in\mathcal{P}_N\setminus \mathcal{P}} &\,\P(\textsc{bin}(r,\tfrac{1}{p})\le 1)\cdot \prod_{j=1}^r \prod_{p\in \mathcal{P}_j} \frac{1}{p^\alpha}\, \Big(1-\frac{1}{p}\Big)^{r-1}
\\
&=
\frac{1}{|\mathbf{a}|}\, \prod_{p\in\mathcal{P}_N\setminus \mathcal{P}} \,\P(\textsc{bin}(r,\tfrac{1}{p})\le 1)\cdot \prod_{j=1}^r \prod_{p\in \mathcal{P}_j} \Big(1-\frac{1}{p}\Big)^{r-1}
\\
& =\frac{1}{|\mathbf{a}|}\, \prod_{p\in\mathcal{P}_N} \,\P(\textsc{bin}(r,\tfrac{1}{p})\le 1)\cdot \prod_{j=1}^r \prod_{p\in \mathcal{P}_j} \frac{(1-{1}/{p})^{r-1}}{(1-1/p)^{r}+r/p(1-1/p)^{r-1}}
\\
& =\frac{1}{|\mathbf{a}|}\, \prod_{p\in\mathcal{P}_N} \,\P(\textsc{bin}(r,\tfrac{1}{p})\le 1)\cdot \prod_{j=1}^r \prod_{p\in \mathcal{P}_j} \frac{1}{1+(r-1)/p}
\end{align*}

Recalling the definition \eqref{eq:Psi function} of the $\Psi$ function, and as $N$ is arbitrary, we get that
$$
\limsup_{n\to\infty}\ \P\big(\mathbf{X}^{(n)}\in\mathcal{PC}^\mathbf{a}\big) \le T_r\, \frac{1}{\Psi_{r-1}(a_1)\cdots \Psi_{r-1}(a_r)}.
$$
The usual argument with the tail finishes the proof of \eqref{eq:PC with extra divisibility}.

\smallskip
b) For mutual coprimality, the random variable of interest is
$$
\prod_{p\in \mathcal{P}_N\setminus \mathcal{P}} \uno_{\{\sum_{j=1}^r I_p(X_j^{(n)})\le r-1\}}\cdot  \prod_{j=1}^r \prod_{p\in\mathcal{P}_j} I_{p^\alpha}(X_j^{(n)})\, \uno_{\{\sum_{i\ne j} I_p(X_i^{(n)})\le r-2\}},
$$
and the calculation of probabilities goes like this:
\begin{align*}
&\prod_{p\in\mathcal{P}_N\setminus \mathcal{P}} \P(\textsc{bin}(r,\tfrac{1}{p})\le r-1)\cdot \prod_{j=1}^r \prod_{p\in \mathcal{P}_j} \frac{1}{p^\alpha}\, \P(\textsc{bin}(r-1,\tfrac{1}{p})\le r-2)
\\
&\quad=\frac{1}{|\mathbf{a}|}\,
\prod_{p\in\mathcal{P}_N\setminus \mathcal{P}} \,\Big(1-\frac{1}{p^r}\Big)\cdot \prod_{j=1}^r \prod_{p\in \mathcal{P}_j} \Big(1-\frac{1}{p^{r-1}}\Big)
\\
& \quad=\frac{1}{|\mathbf{a}|}\, \prod_{p\in\mathcal{P}_N} \,\Big(1-\frac{1}{p^r}\Big)\cdot \prod_{j=1}^r \prod_{p\in \mathcal{P}_j} \frac{(1-{1}/{p^{r-1}})}{1-1/p^r}
 =\frac{\varphi_{r-1}(|\mathbf{a}|)}{\varphi_{r}(|\mathbf{a}|)}\prod_{p\in\mathcal{P}_N} \,\Big(1-\frac{1}{p^r}\Big),
\end{align*}
using the definition \eqref{eq:Jordan function} of the Jordan function $\varphi_r$. The proof finishes as before.
\end{proof}

\subsection{Coprimality and arithmetic progressions}\label{subsection:coprimality and progressions}
As a natural extension of the previous result, we analyze the proportion of $r$-tuples of integers that are pairwise (or mutually) coprime when restricting to arithmetic progressions.

\smallskip
Again, fix $\mathbf{a}=(a_1,\dots, a_r)\in \mathcal{PC}$, and now add an $r$-tuple $\mathbf{b}=(b_1,\dots, b_r)$, where $0\le b_j\le a_j-1$ for each $j=1,\dots, r$. We want to estimate the proportion of $r$-tuples of integers that are (pairwise/mutually) coprime when each coordinate $x_j$ satisfies that $x_j\equiv b_j$ (mod $a_j$). The case $\mathbf{b}=\mathbf{0}$ is the one treated in the previous section.


\smallskip
We write $\mathbf{x}\in \mathcal{PC}^\mathbf{a, b}$ (or $\mathbf{x}\in \mathcal{C}^\mathbf{a, b}$) if $\mathbf{x}\in \mathcal{PC}$ (or $\mathbf{x}\in \mathcal{C}$) and, additionally, $a_i| x_i-b_i$ for all $i=1,\dots, r$ (abbreviated, $\mathbf{a}|\mathbf{x}-\mathbf{b}$).

\begin{theorem}\label{thm:C and PC on AP}
Given $\mathbf{a}\in \mathcal{PC}$ and $\mathbf{b}\in \mathbb{N}^r$, we have:

{\rm a)} For pairwise coprimality,
\begin{align}
\nonumber
\lim_{n \to \infty}\P\big(&\,\mathbf{X}^{(n)} \in\mathcal{PC}^\mathbf{a,b}\big)
\\
&=T_r\, \frac{\Psi_{r-2}(|\mathbf{a}|)}{\Psi_{r-1}(|\mathbf{a}|)}
\frac{1}{\varphi(|\mathbf{a}|)}\,
\prod_{i=1}^r
\frac{\varphi(\gcd(a_i,b_i))}
{\Psi_{r-2}(\gcd(a_i,b_i))}.
\label{eq:PC on AP}
\end{align}

\smallskip
{\rm b)} For mutual coprimality,
\begin{align}
\nonumber
\lim_{n \to \infty}\P\big(&\,\mathbf{X}^{(n)} \in\mathcal{C}^\mathbf{a,b}\big)
\\
&=\frac{1}{\zeta(r)}\ \frac{1}{|\mathbf{a}|}\ \frac{|\mathbf{a}|^r}{\varphi_{r}(|\mathbf{a}|)}\
\prod_{i=1}^r
\frac{\varphi_{r-1}(\gcd(a_i,b_i))}
{\gcd(a_i,b_i)^{r-1}}.
\label{eq:C on AP}
\end{align}
\end{theorem}

\begin{proof}
a) As usual, denote $\mathcal{P}_1,\dots, \mathcal{P}_r$ the disjoint sets of primes dividing $(a_1,\dots, a_r)\in \mathcal{PC}$, and write $\mathcal{P}=\cup_{j=1}^r \mathcal{P}_j$. We need again to keep track of the exponent $\alpha$ of each prime $p$ appearing in the decomposition of the $a_j$.

We partition each $\mathcal{P}_j$ into two subsets:
\begin{itemize}
\item $\mathcal{Q}_j$ contains the primes $p$ of $\mathcal{P}_j$ such that $p|b_j$. Observe that $p|\gcd(a_j,b_j)$.
\item $\mathcal{R}_j$ contains the primes $p$ of $\mathcal{P}_j$ such that $p\nmid b_j$.
\end{itemize}

Fix $N$ large enough so that $\mathcal{P}\subset\mathcal{P}_N:=\{p_1,\dots, p_N\}$. Recall that we want to estimate the probability that $\mathbf{X}^{(n)} \in\mathcal{PC}$ and, additionally,  $a_i| X_i^{(n)}-b_i$ for all $i=1,\dots, r$.

We will write the argument for the first coordinate. Say that $p$ (with exponent $\alpha$) belongs to $\mathcal{P}_1$. We need that $p^\alpha| X_1^{(n)}-b_1$ and that $p$ divides, at most, one of the $X_j^{(n)}$. If $p\in \mathcal{Q}_1$, then~$p| \gcd(a_1,b_1)$, and so $p$ divides $X_1^{(n)}$. On the other hand, if $p\in \mathcal{R}_1$, $p$ does not divide~$X_1^{(n)}$.
So the random variable registering the conditions for the first coordinate can be written as
$$
\prod_{p\in\mathcal{Q}_1} \, I_{p^\alpha}(X_1^{(n)}-b_1)\, \uno_{\{\sum_{i\ne 1} I_p(X_i^{(n)}=0\}}\cdot \prod_{p\in\mathcal{R}_1} \, I_{p^\alpha}(X_1^{(n)}-b_1)\, \uno_{\{\sum_{i\ne 1} I_p(X_i^{(n)}=1\}}.
$$
The corresponding product of probabilities will be
$$
\prod_{p\in\mathcal{Q}_1} \, \frac{1}{p^\alpha}\Big(1-\frac{1}{p}\Big)^{r-1} \cdot \prod_{p\in\mathcal{R}_1} \, \frac{1}{p^\alpha}\, \P(\textsc{bin}(r-1,\tfrac{1}{p})\le 1).
$$
Notice that the presence of $b_1$ does not change the probability $1/p^{\alpha}$.

Putting all the primes together we get
\begin{align*}
&\prod_{p\in\mathcal{P}_N\setminus \mathcal{P}} \P(\textsc{bin}(r,\tfrac{1}{p})\le 1)\cdot \prod_{j=1}^r \Big(\prod_{p\in\mathcal{Q}_j} \, \frac{1}{p^\alpha}\Big(1-\frac{1}{p}\Big)^{r-1} \prod_{p\in\mathcal{R}_j} \, \frac{1}{p^\alpha}\, \P(\textsc{bin}(r-1,\tfrac{1}{p})\le 1)\Big)
\\
&=
\frac{1}{|\mathbf{a}|}\, \prod_{p\in\mathcal{P}_N\setminus \mathcal{P}} \,\P(\textsc{bin}(r,\tfrac{1}{p})\le 1)\cdot \prod_{j=1}^r \Big(\prod_{p\in\mathcal{Q}_j} \, \Big(1-\frac{1}{p}\Big)^{r-1} \prod_{p\in\mathcal{R}_j}
\, \P(\textsc{bin}(r-1,\tfrac{1}{p})\le 1)\Big)
\\
& =\frac{1}{|\mathbf{a}|}\, \prod_{p\in\mathcal{P}_N} \,\P(\textsc{bin}(r,\tfrac{1}{p})\le 1)
\\
&\qquad\cdot \prod_{j=1}^r \Big(\prod_{p\in\mathcal{Q}_j} \, \frac{(1-\frac{1}{p})^{r-1}}{(1-\frac{1}{p})^r+\frac{r}{p}(1-\frac{1}{p})^{r-1}} \prod_{p\in\mathcal{R}_j}
\, \frac{(1-\frac{1}{p})^{r-1}+\frac{r-1}{p} (1-\frac{1}{p})^{r-2}}{(1-\frac{1}{p})^r+\frac{r}{p}(1-\frac{1}{p})^{r-1}}
\\
& =\frac{1}{|\mathbf{a}|}\, \prod_{p\in\mathcal{P}_N} \,\P(\textsc{bin}(r,\tfrac{1}{p})\le 1)
\cdot
\prod_{j=1}^r \Big(\prod_{p\in\mathcal{Q}_j} \, \frac{1}{1+\frac{r-1}{p}} \prod_{p\in\mathcal{R}_j}
\, \frac{1+\frac{r-2}{p}}{(1-\frac{1}{p})(1+\frac{r-1}{p})}\Big)
\\
& =\frac{1}{|\mathbf{a}|}\, \prod_{p\in\mathcal{P}_N} \,\P(\textsc{bin}(r,\tfrac{1}{p})\le 1)
\cdot
\prod_{j=1}^r \Big(\prod_{p\in\mathcal{P}_j} \, \frac{1}{1+\frac{r-1}{p}} \prod_{p\in\mathcal{R}_j}
\, \frac{1+\frac{r-2}{p}}{1-\frac{1}{p}}\Big)
\\
& =\frac{1}{\Psi_{r-1}(|\mathbf{a}|)}\, \prod_{p\in\mathcal{P}_N} \,\P(\textsc{bin}(r,\tfrac{1}{p})\le 1)
\cdot
\prod_{j=1}^r \Big(\prod_{p\in\mathcal{R}_j}
\, \frac{1+\frac{r-2}{p}}{1-\frac{1}{p}}\Big),
\end{align*}
where, in the last step, we have used the definition \eqref{eq:Psi function} of the $\Psi$ function. Now, on the one hand,
\begin{align*}
\prod_{p\in\mathcal{R}_j}
\, \Big(1+\frac{r-2}{p}\Big)&=
\prod_{p\in\mathcal{P}_j}
\, \Big(1+\frac{r-2}{p}\Big)\cdot \prod_{p\in\mathcal{R}_j}\frac{1}{1+\frac{r-2}{p}}=\frac{\Psi_{r-2}(a_j)}{a_j}\, \frac{\gcd(a_j,b_j)}{\Psi_{r-2}(\gcd(a_j,b_j))},
\end{align*}
recalling that if $p\in \mathcal{R}_j$ then $p|\gcd(a_j,b_j)$ and the definition of $\Psi$. On the other hand,
\begin{align*}
\prod_{p\in\mathcal{R}_j}
\, \frac{1}{1-1/p}&=
\prod_{p\in\mathcal{P}_j}
\, \frac{1}{1-1/p}\cdot \prod_{p\in\mathcal{R}_j}\Big(1-\frac{1}{p}\Big)=\frac{a_j}{\varphi(a_j)}\, \frac{\varphi(\gcd(a_j,b_j))}{\gcd(a_j,b_j)},
\end{align*}
We deduce \eqref{eq:PC on AP} with the usual arguments.

\medskip

b) It follows the same lines. Now the random variable for the first coordinate is
$$
\prod_{p\in\mathcal{Q}_1} \, I_{p^\alpha}(X_1^{(n)}-b_1)\, \uno_{\{\sum_{i\ne 1} I_p(X_i^{(n)}\le r-2\}}\cdot \prod_{p\in\mathcal{R}_1} \, I_{p^\alpha}(X_1^{(n)}-b_1)\, \uno_{\{\sum_{i\ne 1} I_p(X_i^{(n)}\le r-1\}},
$$
and the corresponding product of probabilities will be
\begin{align*}
\prod_{p\in\mathcal{Q}_1} \, \frac{1}{p^\alpha}\Big(1-\frac{1}{p^{r-1}}\Big) \cdot \prod_{p\in\mathcal{R}_1} \, \frac{1}{p^\alpha}
\end{align*}
All together, we get
\begin{align*}
\prod_{p\in\mathcal{P}_N\setminus \mathcal{P}} &\,\P(\textsc{bin}(r,\tfrac{1}{p})\le r-1)\cdot \prod_{j=1}^r \Big(\prod_{p\in\mathcal{Q}_j} \, \frac{1}{p^\alpha}\Big(1-\frac{1}{p^{r-1}}\Big) \cdot \prod_{p\in\mathcal{R}_j} \, \frac{1}{p^\alpha}\Big)
\\
&\quad=\frac{1}{|\mathbf{a}|}\,
\prod_{p\in\mathcal{P}_N\setminus \mathcal{P}} \,\Big(1-\frac{1}{p^r}\Big)\cdot \prod_{j=1}^r \prod_{p\in \mathcal{Q}_j} \Big(1-\frac{1}{p^{r-1}}\Big)
\\
& \quad=\frac{1}{|\mathbf{a}|}\, \prod_{p\in\mathcal{P}_N} \,\Big(1-\frac{1}{p^r}\Big)\cdot \prod_{j=1}^r \Big(\prod_{p\in \mathcal{Q}_j} \frac{(1-{1}/{p^{r-1}})}{1-1/p^r}
\prod_{p\in \mathcal{R}_j} \frac{1}{1-1/p^r}\Big).
\end{align*}
Finally observe that
\begin{align*}
\prod_{p\in \mathcal{Q}_j} \frac{(1-{1}/{p^{r-1}})}{1-1/p^r}
\prod_{p\in \mathcal{R}_j} \frac{1}{1-1/p^r}&=
\prod_{p\in \mathcal{P}_j} \frac{1}{1-1/p^r}
\prod_{p\in \mathcal{Q}_j} \Big(1-\frac{1}{p^{r-1}}\Big)
\\
&
=\frac{\varphi_r(a_j)}{a_j^r}\, \frac{\varphi_{r-1}(\gcd(a_j,b_j))}{\gcd(a_j,b_j)^{r-1}}.
\end{align*}
Once more, \eqref{eq:C on AP} is deduced from here.
\end{proof}

\section{Discrepancies for mutual and pairwise coprimality}\label{section:equidistribution for C and PC}
For the sets of points $\mathcal{C}$ and $\mathcal{PC}$ of $\mathbb{N}^r$ with mutually or pairwise coprime coordinates, there are precise estimates for the discrepancies.

\smallskip
For  $n \ge 1$ and $\boldsymbol{\alpha}\in [0,1]^r$ we write
\begin{equation*}
F(n,\boldsymbol{\alpha})=\#\{\mathbf{x}\le n{\boldsymbol{\alpha}}: \mathbf{x}\in \mathcal{C} \}
\quad \text{and} \quad
G(n,\boldsymbol{\alpha})=\#\{\mathbf{x}\le n{\boldsymbol{\alpha}}: \mathbf{x}\in \mathcal{PC} \}.
\end{equation*}
If $\boldsymbol{1}=(1,1,\ldots, 1)$, then the discrepancy functions of $\mathcal{C}$ and $\mathcal{PC}$, may be written as
\begin{equation*}
\Delta_{\mathcal{C}}(n)=\sup_{\boldsymbol{\alpha}\in[0,1]^r}\,\bigg|\frac{F(n,\boldsymbol{\alpha})}{F(n,\boldsymbol{1})}-|\boldsymbol{\alpha}|\bigg|\quad \text{and} \quad
\Delta_{\mathcal{PC}}(n)=\sup_{\boldsymbol{\alpha}\in[0,1]^r}\,\bigg|\frac{G(n,\boldsymbol{\alpha})}{G(n,\boldsymbol{1})}-|\boldsymbol{\alpha}|\bigg|
\end{equation*}

%
%
\subsection{Discrepancy for mutual coprimality}For mutual coprimality,  we have the following bounds on discrepancy:
\begin{theorem} \label{th:discrepancy for mutual}
For $r=2$, there are constants $0<c_2 <C_2$ such that $$\frac{c_2}{n}\le \Delta_{\mathcal{C}}(n)\le C_2\,\frac{\ln n}{n}\,.$$
For any $r \ge 3$, there are constants $0<c_r <C_r$ such that
$$ \frac{c_r}{n}\le \Delta_{\mathcal{C}}(n)\le \frac{C_r}{n}\,.$$
\end{theorem}
It would be interesting to determine whether the upper bound $\ln(n)/n$ in the case $r=2$ above could be improved or not.
\begin{proof}
We start with the case $r=2$. For  $0\le a,b\le 1$ we may write, thanks to \eqref{eq:use of mu},
\begin{equation}\label{eq:expression_for F}
F(n,(a,b))=\#\{x\le an, y\le bn : \gcd(x,y)=1\}=\sum_{d=1}^{\min(an,bn)} \mu(d)\, \Big\lfloor \frac{an}{d}\Big\rfloor\, \Big\lfloor \frac{bn}{d}\Big\rfloor\,.
\end{equation}
The lower bound of $\Delta_\mathcal{C}(n)$ follows simply by observing that for any  $a<1/n$, $F(n,(a,1))=0$.
For the upper bound, rewrite \eqref{eq:expression_for F},  using $\lfloor x\rfloor$ as $x-\{x\}$, to obtain
\begin{align*} \nonumber
F(n,(a,b))
&=n^2\, ab\, \sum_{d=1}^{\min(an,bn)} \frac{\mu(d)}{d^2}+n\sum_{d=1}^{\min(an,bn)} O(1/d)
\\
&=n^2\, ab\, \sum_{d=1}^{\min(an,bn)} \frac{\mu(d)}{d^2} + O(n\ln n)=n^2\, ab\, \frac{1}{\zeta(2)} + O(n\ln n),
\end{align*}
where we have used that $\sum_{d\ge 1} \mu(d)/d^2=1/\zeta(2)$ and that $\sum_{d\ge n+1} 1/d^2=O(1/n)$. Now,
$$
\frac{F(n,(a,b))}{F(n,(1,1))}=\frac{n^2\, ab/{\zeta(2)} + O(n\ln n)}{n^2/{\zeta(2)} + O(n\ln n)}=ab +O\Big(\frac{\ln n}{n}\Big),
$$
as desired.

\smallskip
For $r>2$, we would obtain similarly that
for any $\boldsymbol{\alpha} \in [0,1]^r$
$$
F(n, \boldsymbol{\alpha})=|\boldsymbol{\alpha}| n^r \, \frac{1}{\zeta(r)} +O(n^{r-1}),
$$
giving directly that
$$
\frac{F(n, \boldsymbol{\alpha})}{F(n, 1)}=|\boldsymbol{\alpha}|+O\Big(\frac{1}{n}\Big)\,.
$$
The lower bound follows from observing that $F(n,(a,1,\ldots,1))=0$ for $0<a<1/n$.
\end{proof}

\subsection{Discrepancy for pairwise coprimality}
For pairwise coprimality,  we have the following bounds on discrepancy:
\begin{theorem} \label{th:discrepancy for pairwise}For each $r\ge 2$ there are positive constants $0<c_r<C_r$ such that
\begin{equation}\label{eq:discrepancy for pairwise}
\frac{c_r}{n}\le \Delta_{\mathcal{PC}}(n)\le C_r\,\frac{\ln^{r-1}(n)}{n} .
\end{equation}
\end{theorem}
The proof of this theorem is based on the following extension of Toth's theorem in \cite{To2004}:
\begin{theorem}
Fix integers $r\ge 2$ and $u$, and an $r$-tuple $\mathbf{n}=(n_1,\dots, n_r)$, with $n_j\le n$ for $j=1\dots, r$. Denote
$$
P_r^{(u)}(n_1,\dots, n_r)=\{\mathbf{x}\le \mathbf{n}: \mathbf{x}\in \mathcal{PC}, \gcd(x_i,u)=1 \text{ for $i=1,\dots, r$}\}.
$$
Then
$$
P_r^{(u)}(n_1,\dots, n_r)=T_r\, f_r(u)\, (n_1\cdots n_r) +O(\theta(u)\, n^{r-1}\, \ln^{r-1}(n)),
$$
where $f_r(u)=\prod_{p|u} (1-\frac{r}{p+r-1})$ and $\theta(u)$ is the number of squarefree divisors of $u$.
\end{theorem}
\begin{proof}
It is just a minor modification of the proof in \cite{To2004}. First observe that, by conditioning to the value of the last coordinate,
$$
P_{r+1}^{(u)}(n_1,\dots, n_r, n_{r+1})=\sum_{\substack{1\le t\le n_{r+1}\\ \gcd(t,u)=1}} P_{r}^{(tu)}(n_1,\dots, n_r).
$$
The claim follows by induction, as in \cite{To2004}.
\end{proof}
\begin{proof}[Proof of Theorem {\upshape\ref{th:discrepancy for pairwise}}] The case $u=1$ of the previous theorem gives
$$
G(n, \boldsymbol{\alpha})=|\boldsymbol{\alpha}|\, T_r\, n^r +O(n^{r-1}\, \ln^{r-1}(n)).
$$
Therefore,
$$
\frac{G(n, \boldsymbol{\alpha})}{G(n, \boldsymbol{1})}=|\boldsymbol{\alpha}| +O\Big(\frac{\ln^{r-1}(n))}{n}\Big)\, .
$$
\end{proof}

\subsection{Discrepancies for gcd and lcm}

Consider, for each $n \ge 1$, the measure $\mu_n$ in $[0,1]^2$
\begin{equation*}
\mu_n=\sum_{\substack{ 1\le x\le n,1 \le y\le n,\\ \gcd(x,y)=1}} \,\delta_{(x/n, y/n)}\,.
\end{equation*}
Equidistribution  of the set of totient points in $\mathbb{N}^2$ means that the normalized measure $\widetilde{\mu}_n=\mu_n/\mu_n([0,1]^2)$ converges to Lebesgue measure in $[0,1]^2$.

\medskip

Consider now the measure $\nu_n$ in $[0,1]^2$
\begin{equation}
\label{eq:measure with gcd}
\nu_n=\sum_{1\le x\le n,1 \le y\le n} \gcd(x,y) \,\delta_{(x/n, y/n)},
\end{equation}
which places mass $\gcd(x,y)$ at each point $(x/n,y/n)$.

\begin{proposition}
{\label{lemma:convergence measure gcd}
The probability measure $\widetilde{\nu}_n=\nu_n/\nu_n([0,1]^2)$ converges to  Lebesgue measure in $[0,1]^2$ as $n\to\infty$. In fact,
\begin{equation}
\label{eq:discrepancy for gcd}
\big|\widetilde{\nu}_n([0,a], [0,b])-ab\big|\le C \frac{1}{\ln n}.
\end{equation}
}
\end{proposition}
\begin{proof}
Just write, using again \eqref{eq:use of mu},
\begin{align*}
&{\nu}_n([0,a], [0,b])=\sum_{x\le an, y\le bn} \gcd(x,y)=\sum_{d=1}^{n\min(a,b)} d \cdot \#\{x\le an, y\le bn, \gcd(x,y)=d\}
\\
&\qquad=\sum_{d=1}^{n\min(a,b)} d \cdot \#\{x\le \tfrac{an}{d}, y\le \tfrac{bn}{d}, \gcd(x,y)=1\}
=\sum_{d=1}^{n\min(a,b)} d \sum_{k\ge 1} \mu(k)\Big\lfloor \frac{an}{dk}\Big\rfloor\Big\lfloor \frac{bn}{dk}\Big\rfloor.
\end{align*}
Now, using that $x=\lfloor x\rfloor -\{x\}$, we get
\begin{align*}
{\nu}_n([0,a], [0,b])&=
\sum_{d=1}^{n\min(a,b)} d \sum_{k= 1}^{\min(a,b) n/d}  \mu(k)\Big[\frac{abn^2}{d^2k^2}+O\Big(\frac{n}{dk}\Big)\Big]
\\
&=ab n^2 \sum_{d=1}^{n\min(a,b)} \frac{1}{d}\sum_{k=1}^{n\min(a,b)/d} \frac{\mu(k)}{k^2}+O\Big(n \sum_{d=1}^{n\min(a,b)} \sum_{k=1}^{n\min(a,b)/d} \frac{1}{k}\Big)
\\
&=ab n^2 \sum_{d=1}^{n\min(a,b)} \frac{1}{d}\sum_{k=1}^{n\min(a,b)/d} \frac{\mu(k)}{k^2}+O\Big(n \sum_{d=1}^{n\min(a,b)} \ln(n/d)\Big)
\\
&=ab\, \frac{1}{\zeta(2)} n^2 \ln(n)+ O(n^2).
\end{align*}
This yields \eqref{eq:discrepancy for gcd}.\end{proof}

Let us consider now the measure  $\eta_n$ in $[0,1]^2$ which places mass $\lcm(x,y)$ at each point $(x/n,y/n)$:
\begin{equation*}
\label{eq:measure with lcm}
\eta_n=\sum_{ x\le n,y\le n} \lcm(x,y) \delta_{(x/n, y/n)}.
\end{equation*}
 Recalling that $\lcm(x,y)=xy/\gcd(x,y)$ and following the lines of the argument of Proposition \ref{lemma:convergence measure gcd}, one can see that
$$
\eta_n([0,a], [0,b])=a^2b^2 n^4 \frac{\zeta(3)}{4\zeta(2)} +O(n^3\ln(n)).
$$
This yields:
\begin{proposition}
{\label{lemma:convergence measure lcm}
The probability measure $\widetilde{\eta}_n=\eta_n/\eta_n([0,1]^2)$ does not converge to the Lebesgue measure in $[0,1]^2$, but to the probability measure in $[0,1]^2$ with density $4ab$. Furthermore,
\begin{equation*}
\label{eq:discrepancy for lcm}
\big|\widetilde{\nu}_n([0,a], [0,b])-a^2b^2\big|\le C \,\frac{\ln n}{n}.
\end{equation*}
}
\end{proposition}

\begin{remark}[Higher dimensions]
{\upshape For the gcd, the argument above can be readily extended to higher dimensions, giving that the measure defined placing masses $\gcd(\mathbf{x})$ at points $\mathbf{x}/n$ (where $\mathbf{x}=(x_1,\dots, x_r)$) converges to the Lebesgue measure in $[0,1]^r$. For the lcm, on the other hand, we do not have any results  about the limiting behaviour of the sequence of probability measures
$$
\widetilde{\eta}_n=\frac{\eta_n}{\eta_n([0,1]^r)}, \quad\text{with}\quad
\eta_n=\sum_{ \mathbf{x}\le n} \lcm(\mathbf{x}) \, \delta_{(\mathbf{x}/n)},
$$
except for the case $r=3$, which would be too cumbersome to state here. See \cite{ED2004} and \cite{FF1}.
}
\end{remark}

\

\noindent\textsc{Jos\'{e} L. Fern\'{a}ndez:} Departamento de Matem\'{a}ticas, Universidad Aut\'{o}noma de Madrid, 28049-Madrid, Spain.
\texttt{joseluis.fernandez@uam.es}.

\smallskip

\noindent\textsc{Pablo Fern\'{a}ndez:} Departamento de Matem\'{a}ticas, Universidad Aut\'{o}noma de Madrid, 28049-Madrid, Spain.
\texttt{pablo.fernandez@uam.es}.

 \renewcommand{\thefootnote}{\fnsymbol{footnote}}
\footnotetext{The research of both authors is partially supported by Fundaci\'{o}n Akusmatika. The second named author is partially supported by the Spanish Ministerio de Ciencia e Innovaci\'{o}n, project no. MTM2011-22851.}
\renewcommand{\thefootnote}{\arabic{footnote}}

\end{document}